\documentclass{amsart}

\usepackage{todonotes}
\usepackage{enumerate}
\usepackage{amssymb}
\usepackage{amsmath,amscd}
\usepackage{amsthm}
\usepackage{graphicx}
 \usepackage{verbatim}
 \usepackage{stmaryrd}

\newcommand{\R}{\mathbb{R}}

\newcommand{\OO}{\operatorname{O}}
\newcommand{\SO}{\operatorname{SO}}

\newcommand{\SU}{\operatorname{SU}}
\newcommand{\Sp}{\operatorname{Sp}}
\newcommand{\Spin}{\operatorname{Spin}}

\newcommand{\so}{\mathfrak{so}}

\newcommand{\In}{\subset}

\newcommand{\Hom}{\operatorname{Hom}}

\newcommand{\scal}[1]{\langle #1\rangle}

\newcommand{\Ad}{\operatorname{Ad}}
\newcommand{\Ric}{\operatorname{Ric}}

\newcommand{\Ind}{\operatorname{ind}}
\newcommand{\ACS}{\operatorname{ACS}}
\newcommand{\isom}{\operatorname{Isom}}
\newcommand{\id}{\operatorname{Id}}

\newcommand{\II}{I\!I}

\newcommand{\g}{\mathfrak{g}}
\newcommand{\h}{\mathfrak{h}}
\newcommand{\m}{\mathfrak{m}}
\newcommand{\V}{\mathcal{V}}
\newcommand{\Hr}{\mathcal{H}}

\newtheorem{theorem}{Theorem}

\newtheorem{corollary}[theorem]{Corollary}
\newtheorem*{corollary*}{Corollary}
\newtheorem{lemma}[theorem]{Lemma}
\newtheorem{proposition}[theorem]{Proposition}
\newtheorem*{conjecture}{Conjecture}

\newtheorem{maintheorem}{Theorem}

\newtheorem{maincorollary}[maintheorem]{Corollary}

\theoremstyle{definition}
\newtheorem{definition}[theorem]{Definition}

\theoremstyle{remark}
\newtheorem{remark}[theorem]{Remark}
\newtheorem{example}[theorem]{Example}

\title[Virtual immersions and minimal hypersurfaces]{Virtual immersions and minimal hypersurfaces in compact symmetric spaces}

\author[R.~Mendes]{Ricardo A. E. Mendes$^* \dag$}
\address{University of Oklahoma, US}
\email{ricardo.mendes@ou.edu}

\author[M.~Radeschi]{Marco Radeschi$^{*+}$}
\address{University of Notre Dame, US}
\email{mradesch@nd.edu}

\thanks{$^*$ received support from SFB 878: \emph{Groups, Geometry \& Actions}}
\thanks{$\dag$ received support from DFG ME 4801/1-1 and NSF grant DMS-2005373}
\thanks{$^+$ received support from NSF grant 1810913}


\begin{document}

\maketitle
\begin{abstract}
We show that closed, immersed, minimal hypersurfaces in a compact symmetric space satisfy a lower bound on the index plus nullity, which depends linearly on their first Betti number. Moreover, if either the minimal hypersurface satisfies a certain genericity condition, or if the ambient space is a product of two CROSSes, we improve this to a lower bound on the index alone, which is affine in the first Betti number. To prove these results, we introduce a generalization of isometric immersions in Euclidean space. Compact symmetric spaces admit (and in fact are characterized by) such a structure with \emph{skew-symmetric} second fundamental form.
\end{abstract}

\section{Introduction}

Let $(M,g)$ be a Riemannian manifold, and $\Sigma$ a minimal immersed submanifold. This means that the second fundamental form of $\Sigma$ is traceless, or, equivalently, that $\Sigma$ is a critical point of the area functional. Then one is naturally led to consider variations up to second order, and to define the (Morse)  \emph{index} of $\Sigma$ as the dimension of the space of negative variations. When $\Sigma$ is closed, the index is finite.

Many authors have developed methods to produce minimal submanifolds, including Min-Max Theory (see \cite{Marques14, Neves14} for surveys), desingularization (see for example \cite{Kapouleas11, ChoeSoret16}), and equivariant methods (see for example  \cite{HsiangLawson71, Hsiang83, Hsiang87}). For some of these the index of the minimal submanifold is controlled, while for others the topology is controlled. On the other hand, the set of all minimal submanifolds of bounded index, area, or topology is the object of active research, in particular compactness results such as \cite{ChoiSchoen85, CKM15, Sharp17} have been obtained. Therefore it is natural to ask how the topology and the index of minimal submanifolds are related. One conjecture that fits in this framework is: (see \cite[page 16]{Neves14}, or \cite[page 3]{ACS16} for a slightly different formulation)

\begin{conjecture}[Marques-Neves-Schoen]
Let $(M,g)$ be a compact manifold with positive Ricci curvature, and dimension at least three. Then there exists $C>0$ such that, for all closed embedded orientable minimal hypersurfaces $\Sigma\to M$, 
\[\Ind(\Sigma)\geq Cb_1(\Sigma)\]
where $b_1(\Sigma)$ denotes the first Betti number of $\Sigma$ with real coefficients.
\end{conjecture}
Variations of this conjecture include replacing the assumption that the Ricci curvature is positive with other notions of positivity (or non-negativity) of the curvature; replacing the index with the \emph{extended index} $\Ind_0$, that is, the sum of index and nullity; and replacing the linear bound with an affine bound of the form $\Ind\geq C(b_1-D)$.

Some special cases of the conjecture above (or variations thereof) have been recently established. For example, Ros has considered the case where $(M,g)$ is a flat $3$-torus, and has found affine bounds on the index --- see Theorem 16 in  \cite{Ros06} (see also \cite{ChodoshMaximo16}). The authors of \cite{ACS17} have extended Ros' work to the case where the ambient space $M$ is a flat torus of arbitrary dimension. Namely, they provide an affine bound for the index of minimal hypersurfaces which (if the torus has dimension $>4$) are required to have points where all principal curvatures are distinct. Savo \cite{Savo10} has given linear bounds on the index of minimal hypersurfaces in round spheres, and \cite{ACS16} have extended these bounds to the other compact rank one symmetric spaces. Moreover, the methods in \cite{ACS16} sometimes allow for small perturbations of the ambient metric in certain directions.

Note that the results mentioned above mostly apply to ambient spaces in subclasses of compact symmetric spaces. Our main result applies uniformly to this whole class:
\begin{maintheorem}
\label{MT:indexbound}
Let $(M,g)$ be a compact symmetric space, $G$ its isometry group, and $\Sigma\subset M$ a closed, immersed minimal hypersurface. Then the extended index of $\Sigma$ satisfies
\[ \Ind_0(\Sigma)\geq \binom{\dim G}{2}^{-1}b_1(\Sigma). \]
\end{maintheorem}

To prove Theorem \ref{MT:indexbound} (as well as the previous results mentioned above) one needs to produce enough negative variations, and roughly speaking, these come from coordinates of vector fields. In \cite{Ros06}, \cite{ACS17} about flat tori, the tangent bundle is trivial, and a choice of parallelization leads to such coordinates. In \cite{Savo10}, \cite{ACS16}, such coordinates come from an embedding of the ambient manifold $(M,g)$ into Euclidean space, an idea that goes back at least to \cite{Simons68} (see also  \cite[Corollary 2.2]{Savo10}). Our method of proof generalizes all of these: we consider embeddings of the tangent bundle of $M$ into a flat trivial bundle $M\times V$ over $M$, such that the natural flat connection on $M\times V$ induces the Levi-Civita connection of $M$.

Such structures, which we call \emph{virtual immersions}, exhibit an extrinsic geometry similar to the classical case. More precisely, one may define the normal bundle, second fundamental form, and normal connection, and these satisfy identities analogous to the fundamental equations of Gauss, Codazzi, and Ricci. The important difference is that the second fundamental form is not necessarily symmetric, and in fact the case where it is symmetric corresponds exactly to classical isometric immersions into Euclidean space. In the present article, we mostly consider the opposite extreme, namely virtual immersions with \emph{skew-symmetric} second fundamental form. We show that every
compact symmetric space admits a natural such virtual immersion, which lies at the heart of the proof of Theorem \ref{MT:indexbound}.

By the Nash Embedding Theorem, every Riemannian manifold admits an isometric embedding into Euclidean space. In contrast, virtual immersions with skew-symmetric second fundamental form are extremely rigid, and in fact their existence \emph{characterizes} symmetric spaces:
\begin{maintheorem}
\label{MT:classification}
Let $(M,g)$ be a compact Riemannian manifold. It admits a virtual immersion $\Omega$  with skew-symmetric second fundamental form if and only if it is a symmetric space. In this case, $\Omega$ is essentially unique.
\end{maintheorem}

Let $(M,g)$ be a compact symmetric space. In some situations, one may ``improve'' Theorem \ref{MT:indexbound} to obtain linear or affine bounds on the index, instead of the extended index, of closed immersed minimal hypersurfaces in $M$. For example when $M$ is a CROSS, we recover, in a uniform way, linear bounds for the index, although with worse constants than the ones obtained in \cite{ACS16} --- see Corollary \ref{C:CROSS}. In higher rank, we have:

\begin{maintheorem}
\label{MT:affine}
Let $M=G/H$ be a compact symmetric space of rank $r\geq 2$, with $G=\isom(M)$, and $\Sigma\subset M$ a closed, immersed minimal hypersurface. Then an affine bound of the form
\[ \Ind(\Sigma)\geq \binom{\dim G}{2}^{-1} (b_1(\Sigma)-D) \]
holds in the following cases:
\begin{enumerate}[a)]
\item the hypersurface $\Sigma$ contains a point where all principal curvatures are distinct, and $D=2r-3+\dim\mathfrak{z}(\h)$. Here $\h$ denotes the Lie algebra of $H$, and $\mathfrak{z}(\h)$ its center.
\item $M$ is a product of two CROSSes $M=M_1\times M_2$, and $D$ is one plus the number of two-dimensional factors.

\end{enumerate}
\end{maintheorem}

Both Theorem \ref{MT:affine} and Corollary \ref{C:CROSS} are special cases of a more general, albeit technical, result --- see Theorem \ref{T:rigidity2}.

Part (a) of Theorem \ref{MT:affine} generalizes the main result of \cite{ACS17} from tori to compact symmetric spaces. Part (b) may be compared with \cite[Theorems 10,11]{ACS16}, which provide a linear bound for the index of closed minimal hypersurfaces of products of two spheres $S^a\times S^b$ with $(a,b)\neq (2,2)$.

If the hypersurface $\Sigma$ is unstable, then an affine bound  of the form $\Ind\geq C(b_1-D)$ trivially implies the linear bound $\Ind\geq \frac{C}{1+CD} b_1$. One situation where $\Sigma$ is necessarily unstable is when $M$ has positive Ricci curvature and $\Sigma$ is two-sided (for example when both $M$ and $\Sigma$ are orientable). In particular, we have:
\begin{maincorollary}
Let $M$ be an orientable compact symmetric space whose universal cover has no Euclidean factors. Then the conclusion of the Marques-Neves-Schoen Conjecture holds if $M$ is a product of two CROSSes, or if $\Sigma$ has a point where the principal curvatures are distinct.
\end{maincorollary}

\subsection*{Acknowledgements} It is a pleasure to thank Lucas Ambrozio for many enlightening discussions during this project, especially regarding Lemma \ref{L:locsym} and Theorem \ref{MT:affine}(b). The final part of this project was carried out while the second-named author visited the University of Cologne. The second-named author wishes to thank Alexander Lytchak for his hospitality during the visit.

\subsection*{Conventions} We will denote by $R$ the curvature tensor, and follow the sign convention in \cite[page 89]{doCarmo}. Namely, 
\[ R(X,Y)Z= \nabla_Y\nabla_X Z - \nabla_X \nabla_Y Z + \nabla_{[X,Y]}Z \]
Shape operators will be defined as in \cite[page 128]{doCarmo}, that is, 
\[S_\eta(X)=-(\nabla_X\eta)^T\]

\section{Virtual immersions and their fundamental equations}

Let $(M,g)$ be a compact Riemannian manifold. We define a generalization of isometric immersions of $(M,g)$ into Euclidean space. Namely,  we consider an isometric embedding of $TM$ into a trivial bundle $M\times V$, such that the natural (flat) connection $D$ of $M\times V$ induces the Levi-Civita connection $\nabla$ on $TM$. To make computations more convenient, we phrase this definition in the following, slightly different way  --- see Proposition \ref{P:equivalentdefinition} for a proof that these two definitions coincide.
\begin{definition}
\label{D:genimm}
Let $(M,g)$ be a Riemannian manifold, and $V$ a finite-dimensional real vector space endowed with an inner product $\left<,\right>$. Let $\Omega$ be a $V$-valued one-form on $M$. We say $\Omega$ is a \emph{virtual immersion} if the following two conditions are satisfied:
\begin{enumerate}[a)]
\item $\left<\Omega_p(X),\Omega_p(Y)\right>=g_p(X,Y)$ for every $p\in M$, and every $X,Y\in T_pM$.
\item $\left<(d\Omega)_p(X,Y),\Omega_p(Z)\right>=0$ for every $p\in M$, and every $X,Y, Z\in T_pM$.
\end{enumerate}
We say two virtual immersions $\Omega_i:TM\to V_i$, $i=1,2$ are equivalent if there is a linear isometry $(V_1,\left<,\right>_1 )\to (V_2,\left<,\right>_2 )$ making the obvious diagram commute.
\end{definition}


\begin{example}
Let $\psi: (M,g)\to V$ be an isometric immersion. Then $\Omega=d\psi$ is a virtual immersion in the above sense.
\end{example}

\begin{example}
Let $\Omega_i:TM\to V_i$ be virtual immersions, for $i=1,2$, and let $a_1,a_2\in C^\infty(M)$ such that $a_1^2+a_2^2=1$ everywhere on $M$. Then the map $\Omega_1\oplus\Omega_2: TM\to V_1\oplus V_2$ given by $v\mapsto (a\Omega_1(v),b\Omega_2(v))$ is again a virtual immersion. This follows from a straight-forward computation.
\end{example}

Given a virtual immersion $\Omega$, we shall identify $TM$ with the image of the map $(p,v)\mapsto (p, \Omega_p(v))$ in $M\times V$.

Condition (a) in Definition \ref{D:genimm} yields a decomposition of the trivial vector bundle $M\times V$ as a direct sum $M\times V= TM\oplus \nu M$ of $TM\subset M\times V$ and its orthogonal complement, the normal bundle $\nu M$. Given $(p,X)\in M\times V$, we shall write $X=X^T+X^\perp$ for the decomposition into the tangent and normal parts.

The natural connection $D$ on $M\times V$ induces connections $D^T$ (respectively  $D^\perp$, the \emph{normal connection}) on $TM$ (resp. $\nu M$), given by $D^T_X Y=(D_X Y)^T$ (resp. $D^\perp_X \eta=(D_X \eta)^\perp$). Here $X,Y$ are vector fields on $M$, while $\eta$ is a section of the normal bundle.

\begin{proposition}
\label{P:equivalentdefinition}
Let $\Omega$ be a $V$-valued one-form on $M$  satisfying condition (a) in Definition \ref{D:genimm}. Then, condition (b) is equivalent to $D^T=\nabla$.
\end{proposition}
\begin{proof}
Recall that
\begin{equation}
\label{E:dOmega}
 d\Omega(X,Y)=
D_X Y-D_Y X-[X,Y] \end{equation}
so that taking the tangent part yields
\[  d\Omega(X,Y)^T= D^T_X Y -D^T_Y X - [X,Y] .\]
Condition (a) implies that $D^T$ is compatible with the metric $g$, and by the above formula Condition (b) is equivalent to $D^T$ being torsion-free. Since these two properties characterize the Levi-Civita connection, the result follows.
\end{proof}

\begin{definition}
Let $\Omega$ be a $V$-valued virtual immersion, $X,Y$ be smooth vector fields on $M$, and $\eta$ a smooth section of $\nu M$. Define the \emph{second fundamental form} of $\Omega$ by
\[ \II(X,Y)=
(D_X Y)^\perp
 \]
 and the \emph{shape operator} in the direction of a normal vector $\eta$ by
 \[ S_\eta(X)=-(D_X \eta)^T.\]
\end{definition}
Note that the second fundamental form and the shape operator are tensors. In view of Proposition \ref{P:equivalentdefinition}, we may write
\begin{align}
D_X Y &=\nabla_X Y+ \II(X,Y) \\
D_X \eta &= -S_\eta X +D^\perp_X \eta
\end{align}

\begin{remark}
If $\Omega$ is a virtual immersion, then its second fundamental form is \emph{symmetric} if and only if $d\Omega=0$, or, equivalently, $\Omega$ locally comes from an isometric immersion of $M$ into Euclidean space. Indeed, by  \eqref{E:dOmega}, the normal part of $d\Omega$ equals $\II(X,Y) -\II(Y,X)$. 
\end{remark}

The fundamental equations of the extrinsic geometry of submanifolds of Euclidean space  carry over in similar form to virtual immersions. In fact, following almost verbatim the computations in, for example, \cite[Ch. 6.3]{doCarmo}, one gets the following.
\begin{proposition}
\label{P:fundamental}
Let $\Omega$ be a virtual immersion of the Riemannian manifold $(M,g)$ with values in $V$. Then the following identities hold:
\begin{enumerate}[a)]
\item Weingarten's equation
 \[ \left< S_\eta(X), Y \right>= \left< \II(X,Y),\eta \right> \]
\item Gauss' equation
\[R(X,Y,Z,W)=\left<\II(Y,W),\II(X,Z)\right>-\left<\II(X,W),\II(Y,Z)\right>\]
\item Ricci's equation
\[  \left< R^\perp(X,Y)\eta,\zeta\right>=-\left<(S_\eta^tS_\zeta -S_\zeta^t S_\eta)X,Y\right>\]
\item Codazzi's equation
\[ \left<(D_X \II) (Y,Z) ,\eta \right>= \left<(D_Y \II) (X,Z) ,\eta \right>. \]
\end{enumerate}
\end{proposition}

\section{Index of minimal hypersurfaces}

In this section we show that the method of proof used in \cite{Savo10,ACS16,ACS17} applies not only to immersions of the ambient manifold $M$ into Euclidean space, but also to virtual immersions. The statements that we need, along with their proofs, are essentially the same as in the classical case. We include them here for the sake of completeness and to fix notations. 

Let $(M,g)$ be a Riemannian manifold, and $\Sigma\to M$ be a closed minimal immersed hypersurface. Recall that the Jacobi operator $J_\Sigma$ is the self-adjoint operator on the space $\Gamma(\nu\Sigma)$ of  sections of the normal bundle of $\Sigma$, and it is defined by:
\begin{equation}
J_\Sigma(X)= \Delta^{\perp} X +\big(|A|^2+\Ric^M(N,N)\big)X
\nonumber
\end{equation}
where $N$ is a choice of (locally defined) unit normal vector field to $\Sigma$, $\Delta^\perp$ is the normal Laplacian, and $A$ is the second fundamental form of the immersion $\Sigma\to M$.
The Morse index of $\Sigma$ (resp. the nullity of $\Sigma$) is the index (resp. the dimension of the kernel) of the quadratic form
\[ Q(X,X)=-\int_\Sigma J_\Sigma(X)\cdot X= \int_\Sigma |\nabla^{\perp} X|^2-\big( |A|^2+\Ric^M(N,N) \big) |X|^2.\]

The next lemma is equivalent to \cite[Proposition 3]{ACS16} (see also \cite{Savo10} and \cite[Theorem 16]{Ros06}).
\begin{lemma}\label{L:cond}
Suppose $\Hr$ is a vector space of dimension $b$, and let
\[
X:\Hr\to \Gamma(\nu\Sigma)^\ell,\qquad X(\omega)=(X_1(\omega),\ldots X_\ell(\omega))
\]
be a linear map such that for all $\omega\in \Hr$,
\begin{equation}\label{E:cond}
\sum_{i=1}^\ell Q(X_i(\omega),X_i(\omega))\leq 0\,\, (\textrm{resp.} <0).
\end{equation}
Then $\Ind_0(\Sigma)\geq {1\over \ell}b$ (resp. $\Ind(\Sigma)\geq {1\over \ell}b$).
\end{lemma}

\begin{proof}
Let $E^m\subset \Gamma(\nu\Sigma)$ be the sum of eigenspaces of $J_\Sigma$ with non-positive (resp. negative) eigenvalue, and let
\[
\Phi:\Hr\to \Hom(E,\R^\ell),\quad \omega\mapsto\Big(Y\mapsto \big(\left< Y,X_1(\omega)\right>_{L^2},\ldots \left< Y,X_\ell(\omega)\right>_{L^2}\big)\Big)
\]
where $\left< X,Y\right>_{L^2}:=\int_{\Sigma}\left< X,Y\right>$. Notice that $m=\Ind_0(\Sigma)$ (resp. $m=\Ind(\Sigma)$) and in either case one must prove $b\leq \ell m$.

By contradiction, if $b>\ell m=\dim \Hom(E,\R^\ell)$, then by dimension reasons $\ker\Phi\neq0$ and, given $\omega\in \ker \Phi$ nonzero, it follows that $X_i(\omega)\perp E$ for all $i=1,\ldots \ell$. Therefore $Q(X_i(\omega),X_i(\omega))>0$ (resp. $Q(X_i(\omega),X_i(\omega))\geq 0$) and, taking the sum over all $i=1,\ldots \ell$ one gets the desired contradiction with equation \eqref{E:cond}.
\end{proof}

Suppose $M$ is endowed with a virtual immersion $\Omega:TM\to V$, with second fundamental form $\II$. For any point $p\in M$ and vectors $x,y\in T_pM$, define
\begin{align}
\label{E:ACSqty}
\ACS(x,y):=& |y|^2\operatorname{tr}\Big(|\II(\cdot\,,x) |^2-R(\cdot\,,x,\cdot\,,x)\Big)+  |x|^2\operatorname{tr}\Big(|\II(\cdot\,,y) |^2-R(\cdot\,,y,\cdot\,,y)\Big)\\
&-\Big(|\II(x,y)|^2-R(x,y,x,y)\Big)- |x|^2 |\II(y,y)|^2\nonumber 
 \end{align} 
This quantity appears naturally in the proof of Proposition \ref{P:criterion}, more specifically in equation \eqref{E:ACS}. 

The next result is equivalent (in the case of classical immersions) to Proposition 2 in \cite{ACS16}:
\begin{proposition}\label{P:criterion}
Suppose $M$ admits a virtual immersion $\Omega:TM\to V$, $\dim V=d$, such that, for every point $p\in M$ and every $x,y\in T_pM$ orthonormal vectors, $ACS(x,y)\leq 0$ (resp. $<0$).

Then, for every closed minimal immersed hypersurface $\Sigma\to M$,
\[
\Ind_0(\Sigma)\geq {d\choose 2}^{-1}b_1(\Sigma) \quad\left(\textrm{resp. }\Ind(\Sigma)\geq {d\choose 2}^{-1}b_1(\Sigma)\right).
\]
\end{proposition}
\begin{proof}
Let $\Sigma\to M$ be a closed, immersed, minimal hypersurface. Also, let $\theta_1,\ldots \theta_d$ denote an orthonormal basis of $V$.

Locally around every point of $\Sigma$, it is possible to choose a unit normal vector $N$, which is unique up to sign. Given indices $1\leq i<j\leq d$ and a harmonic 1-form $\omega$ on $\Sigma$, let $\omega^\#$ denote the vector field on $\Sigma$ such that $\scal{\omega^\#,Y}=\omega(Y)$ for any vector field $Y$ in $\Sigma$, and define
\begin{equation}
\label{E:variation}
X_{ij}(\omega):=\left< \omega^\#\wedge N, \theta_i\wedge \theta_j\right> N.
\end{equation}
Notice that the definition of $X_{ij}(\omega)$ does not depend on the specific choice of unit normal vector $N$, and therefore it defines a global section of $\nu\Sigma$, even when $\Sigma$ is 1-sided and there is no global unit vector field $N$ defined on the whole of $\Sigma$.

Letting $\Hr$ denote the space of harmonic 1-forms on $\Sigma$ and letting $\ell={d\choose 2}$, this defines a linear map
\[
X:\Hr\to \Gamma(\nu\Sigma)^\ell, \quad X(\omega)=\big(X_{ij}(\omega)\big)_{i,j}.
\]
The idea is to apply Lemma \ref{L:cond} to get the result. In order to do this, we must compute $\sum_{i<j}Q(X_{ij}(\omega),X_{ij}(\omega))$:
\begin{align*}
Q(X_{ij}(\omega),X_{ij}(\omega))=&\int_\Sigma |\nabla^\perp X_{ij}(\omega)|^2-\big( |A|^2+\Ric^M(N,N) \big) |X_{ij}(\omega)|^2\\
=&\sum_{k=1}^{n-1}\int_\Sigma \left<D_{e_k}(\omega^\#\wedge N), \theta_i\wedge \theta_j\right>^2\\
&-\int_\Sigma\big( |A|^2+\Ric^M(N,N) \big)\left<\omega^\#\wedge N, \theta_i\wedge \theta_j\right>^2
\end{align*}
where $e_1,\ldots e_{n-1}$ denotes a (local) orthonormal frame of $T\Sigma$. Summing over all indices $i, j$, one obtains
\begin{align*}
\sum_{i<j}Q(X_{ij}(\omega),X_{ij}(\omega))=&\sum_{k=1}^{n-1}\int_\Sigma |D_{e_k}(\omega^\#\wedge N)|^2-\big( |A|^2+\Ric^M(N,N) \big) |\omega^\#\wedge N|^2.
\end{align*}
Using the equality $D_{e_i}(\omega^\#\wedge N)=(D_{e_i}\omega^\#)\wedge N+ \omega^\#\wedge (D_{e_i} N)$ and $D_{e_i}=\nabla_{e_i}+\II(e_i, \cdot\,)$, the computations from this point on follow verbatim those in Proposition 2 of \cite{ACS16}. These show that
\begin{equation}\label{E:ACS}
\sum_{i<j}Q(X_{ij}(\omega),X_{ij}(\omega))=\int_\Sigma \ACS(\omega^\#,N)
\end{equation}
which is non-positive (resp. negative) by assumption, and the result now follows from Lemma \ref{L:cond}.
\end{proof}

When the ACS quantity \eqref{E:ACSqty} is non-positive, Proposition \ref{P:criterion} yields a linear bound on the extended index. Sometimes one may also obtain a lower bound on the index, which is in general affine in $b_1$ instead of linear. This is described in the following ``rigidity'' statement: (compare with the proof of Theorem 1.1 in \cite[page 8]{ACS17})
\begin{proposition}\label{P:affinebound}
	Suppose $M$ admits a virtual immersion $\Omega:TM\to V$, $\dim V=d$, such that for every point $p\in M$ and every $x,y\in T_pM$ orthonormal vectors, $ACS(x,y)\leq 0$. Let $\Sigma\to M$ be a closed minimal immersed hypersurface, and let $D$ denote the dimension of the space of harmonic one-forms $\omega$ on $\Sigma$ such that $J_\Sigma (X_{ij}(\omega))=0$ for all $i,j$, where $X_{ij}(\omega)$ is defined in \eqref{E:variation}. Then
	\[
	\Ind(\Sigma)\geq {d\choose 2}^{-1}(b_1(\Sigma)-D) .
	\]
\end{proposition}
\begin{proof}
This proof is similar to the proofs of Lemma \ref{L:cond} and Proposition \ref{P:criterion}, so we use the same notations and only indicate the necessary modifications.

Let $\Hr$ be the space of harmonic $1$-forms on $\Sigma$, and $\Hr'\subset \Hr$ the orthogonal complement to the space of harmonic $1$-forms $\omega$ such that $J_\Sigma(X_{ij}(\omega))=0$ for all $i,j$. Thus $\dim(\Hr')=b_1-D$.

Let $m=\Ind(\Sigma)$, $\ell={d\choose 2}$ and consider the restriction of $\Phi:\Hr\to \Hom(E,\R^\ell)$ to $\Hr'$, where $E$ denotes the space spanned by the eigenfunctions of $J_\Sigma$ associated to negative eigenvalues. Assuming for a contradiction that $b_1-D>\ell m$ yields a non-zero $\omega\in\Hr'$ such that $\Phi(\omega)=0$. Then $Q(X_{ij}(\omega),X_{ij}(\omega))\geq 0$ for all $i,j$, and, since $\ACS\leq 0$, we have $Q(X_{ij}(\omega),X_{ij}(\omega))=0$ for all $i,j$. But $X_{ij}(\omega)$ is a linear combination of eigenfunctions with non-negative eigenvalues, so  $J_\Sigma(X_{ij}(\omega))$ vanishes identically for all $i,j$, a contradiction.
\end{proof}

\section{Skew-symmetric second fundamental form}

In this section we define a natural virtual immersion with skew-symmetric second fundamental form associated to any compact symmetric space, and use it to prove Theorem \ref{MT:indexbound}.  Then we show that this is in fact the unique example of a virtual immersion with skew-symmetric second fundamental form, thus proving Theorem \ref{MT:classification}.

We start by fixing some notations (see for example \cite[Chapter 7]{Besse} for general information about symmetric spaces). Let $(M,g)$ be a compact symmetric space, and $p_0\in M$. Choose a closed transitive subgroup $G$ of the isometry group of $M$ such that $(G,H)$ is a symmetric pair, where $H=G_{p_0}$. Denote by $\pi:G\to G/H=M$ the natural projection $g\mapsto \llbracket g \rrbracket=gp_0$.
Choose an $\Ad_G$-invariant metric $\left<,\right>$ on the Lie algebra $\g$ such that $\pi$ is a Riemannian submersion, and let $\m \subset\g$ be the orthogonal complement of $\h$ with respect to this metric. Then $\m$ is isometric to $T_{p_0}M$ via the differential of $\pi$ at the identity $e\in G$; $\h$ and $\m$ are $\Ad_H$-invariant; and they satisfy $[\m,\m]=\h$.

Define $G\times_H \m$ as the quotient of $G\times\m$ by the action of $H$ given by $h.(g,X)=(gh^{-1}, \Ad_h X)$, and denote by $\llbracket g,X \rrbracket$
the image of $(g,X)\in G\times\m$ under the quotient map. $G\times_H \m$ comes with a natural action by $G$, defined by $g'.\llbracket g,X \rrbracket=\llbracket g'g,X \rrbracket$. Identify the tangent bundle $TM$ with $G\times_H\m$ by extending the isomorphism $\m\to T_{p_0} M$ to the $G$-equivariant isomorphism
\[ \llbracket g,X\rrbracket\mapsto dg(X).\]
With this identification, we define a $\g$-valued one-form $\Omega_0$ on $M$ by
\begin{equation}
\label{E:naturalgenimm}
\Omega_0(\llbracket g,X \rrbracket)=\Ad_gX.
\end{equation}

\begin{lemma}
\label{L:Omega0}
The $\g$-valued one-form $\Omega_0$ defined in Equation \eqref{E:naturalgenimm} is a virtual immersion. At $\llbracket g\rrbracket\in M$, the tangent and normal spaces are $\Ad_g\m$ and $\Ad_g \h$, respectively. The second fundamental form is given by 
\[ \II\big(\llbracket g,X\rrbracket,\llbracket g,Y\rrbracket\big)=\Ad_g([X,Y]) \]
and the curvature of the normal connection is given by
\[ R^\perp(X,Y)\eta=[[X,Y],\eta] .\]
\end{lemma}
\begin{proof}
It is clear from \eqref{E:naturalgenimm} that the tangent and normal spaces are $\Ad_g\m$ and $\Ad_g \h$. 

Let $X\in\g$. Under the identification of $TM$ with $G\times_H \m$ that we are using, the action field $X^*$ is given by 
\[
X^*\llbracket g \rrbracket= \llbracket g,(\Ad_{g^{-1}}X)_\m \rrbracket
.\]
Indeed, 
\[
 dg^{-1} \left(\left. \frac{d}{dt}\right|_{t=0} \llbracket e^{tX} g \rrbracket \right)
 =\left. \frac{d}{dt}\right|_{t=0}\llbracket g^{-1} e^{tX} g\rrbracket
 =d\pi_e(\Ad_{g^{-1}}X) 
=(\Ad_{g^{-1}}X)_\m.
\]

Given $X,Y\in\g$, we then have
\begin{align*}  
D_{X^*}\Omega_0(Y^*)&=\left.\frac{d}{dt}\right|_{t=0} \Ad_{e^{tX}g}((\Ad_{g^{-1}e^{-tX}}Y)_\m)\\
&=\Ad_g\big([\Ad_{g^{-1}}X,(\Ad_{g^{-1}}Y)_\m]  - (\Ad_{g^{-1}}[X,Y])_\m   \big)
\end{align*}

By $G$-equivariance, it is enough to show that,  for every $X,Y\in\m$, we have $d\Omega_0(X^*,Y^*)^T_{p_0}=0$ and $\II(X,Y)_{p_0}=[X,Y]$. Plugging $g=e$ in the equation above, and using the fact that $[\m,\m]=\h$, we have
\[ D_{X^*}\Omega_0(Y^*)= [X,Y]. \]
The tangent part of this is zero, so that
\[d\Omega_0(X^*,Y^*)^T_{p_0}= D_{X^*}\Omega_0(Y^*)^T_{p_0} - D_{Y^*}\Omega_0(X^*)^T_{p_0} - \Omega_0([X^*,Y^*])_{p_0}= 0-0-0=0\]
which means that $\Omega_0$ is a virtual immersion.

Moreover, 
 $\II(X,Y)_{p_0}=D_{X^*}\Omega_0(Y^*)^\perp_{p_0}= [X,Y]$.
 
Finally, we compute the curvature of the normal connection using Ricci's equation (see Proposition \ref{P:fundamental}(c)). From the equation for the second fundamental form above, we see that the shape operator is given by $S_\eta(X)=-[X,\eta]$. Therefore
\begin{align*}
\left< R^\perp(X,Y)\eta , \xi \right> = & \left<[S_\eta,S_\xi] X, Y \right> =
 \left<[[X,\xi],\eta]-[[X,\eta],\xi], Y \right> \\
 = &  \left<-[[\xi,\eta],X], Y \right>= \left<[Y,X], [\xi,\eta] \right> \\
 = &  \left<[[X,Y],\eta], \xi \right>
\end{align*}
where we have used the Jacobi identity in the third equal sign, and bi-invariance of $\left<,\right>$ in the last two equal signs.
\end{proof}

\begin{example}
Let $M=S^{n-1}$, the unit $(n-1)$-sphere. Its isometry group is $\OO(n)$, with Lie algebra $\so(n)$. The latter may be identified with $\wedge^2\R^n$ via the formula $x\wedge y\mapsto xy^t-yx^t$, where $x,y$ are viewed as column $n$-vectors. Take the base point $p_0$ to be the first standard basis vector $(1,0,\ldots,0)^t\in\R^n$. Then the virtual immersion $\Omega_0:TM\to \so(n)$ defined in \eqref{E:naturalgenimm} is simply given by $(p,v)\mapsto p\wedge v$. To prove this, one notes that the map given by this formula and $\Omega_0$ are both $\OO(n)$-equivariant, and that they coincide at the point $p_0\in M$.
\end{example}

\begin{remark}
Geometrically, we may think of the Lie algebra $\g$ as the space of Killing fields on $M$. Then, the map $\Omega_0$ defined in \eqref{E:naturalgenimm} sends the tangent vector $\llbracket g,X \rrbracket\in T_{\llbracket g\rrbracket}M$  to the unique Killing field with this value at $\llbracket g\rrbracket\in M$, and zero covariant derivative at this point. Equivalently, $\Omega_0(\llbracket g,X\rrbracket)$ is the Killing field of smallest
norm (as an element of $\g$) that has the value $\llbracket g,X\rrbracket$ at $\llbracket g\rrbracket$. Indeed, this follows from the formula $X^*\llbracket g\rrbracket= \llbracket g,(\Ad_{g^{-1}}X)_\m \rrbracket$.
\end{remark}

\begin{proof}[Proof of Theorem \ref{MT:indexbound}]
Let $(M,g)$ be a compact symmetric space, and consider $\Omega_0$ defined in Equation \eqref{E:naturalgenimm}. By Lemma  \ref{L:Omega0}, this is a virtual immersion with skew-symmetric second fundamental form. By the Gauss equation (Proposition \ref{P:fundamental}(b)), $R(X, Y, X, Y)=|\II(X,Y)|^2 $, so that in particular the ACS quantity defined in Equation \eqref{E:ACS} vanishes identically. Now the result follows from Proposition \ref{P:criterion}. 
\end{proof}

Next we proceed to the proof of Theorem \ref{MT:classification}. We need the following two lemmas.

\begin{lemma}
\label{L:locsym}
Let $(M,g)$ be a compact Riemannian manifold, and $\Omega$ a $V$-valued virtual immersion with skew-symmetric second fundamental form $\II$. Then:
\begin{enumerate}[a)]
\item $\left< R(X,Y)Z, W\right>=\left< \II(X,Y), \II(Z,W)\right>$.
\item $(D_X \II)(Y,Z)=-R(Y,Z)X$.
\item $\nabla R=0$. In particular, $(M,g)$ is a locally symmetric space.
\end{enumerate}
\end{lemma}

\begin{proof}
\begin{enumerate}[a)]
\item Start with Gauss' equation (see Proposition \ref{P:fundamental}(b)),
\[R(X,Y,Z,W)=\left<\II(Y,W),\II(X,Z)\right>-\left<\II(X,W),\II(Y,Z)\right>\]
Applying the first Bianchi identity yields
\[ 0=-2\big(\left<\II(X,Y),\II(Z,W)\right> + \left<\II(Y,Z),\II(X,W)\right>
+ \left<\II(Z,X),\II(Y,W)\right> \big)\]
so that using Gauss' equation one more time we arrive at \[\left< R(X,Y)Z, W\right>=\left< \II(X,Y), \II(Z,W)\right>.\]
\item First we argue that $(D_X \II)(Y,Z)$ is tangent. Indeed, for any normal vector $\eta$, Codazzi's equation (Proposition \ref{P:fundamental}(d)) says that 
\[ \left<(D_X \II) (Y,Z) ,\eta \right>= \left<(D_Y \II) (X,Z) ,\eta \right>. \]
Thus the trilinear map $(X,Y,Z)\mapsto  \left<(D_X \II) (Y,Z) ,\eta \right>$ is symmetric in the first two entries and skew-symmetric in the last two entries, which forces it to vanish.

Next we let $W$ be any tangent vector and compute
\begin{align*}
 \left<(D_X \II) (Y,Z) ,W \right> &=  \left<D_X ( \II(Y,Z)) ,W \right> =  -\left< \II (Y,Z) ,D_XW \right>\\
 &= -\left< \II (Y,Z) ,\II(X,W) \right> = -\left< R(Y,Z)X ,W \right>
\end{align*}
where in the last equality follows we have used part (a).

\item Since the natural connection $D$ on $M\times V$ is flat, it follows that for any vector fields $X,Y,Z,W$, we have
\[ 0= D_X(D_Y(\II(Z,W))) -D_Y(D_X(\II(Z,W))) -D_{[X,Y]}(\II(Z,W)). \]
Fix $p\in M$, and take vector fields such that $[X,Y]=0$ and $\nabla Z=\nabla W=0$ at $p\in M$. Then, evaluating the equation above at $p\in M$, we have
\begin{align*}
0  = & D_X \big( (D_Y\II) (Z,W) +\II(\nabla_Y Z, W) +\II(Z, \nabla_Y W)\big)\\
& - D_Y \big( (D_X\II) (Z,W) +\II(\nabla_X Z, W) +\II(Z, \nabla_X W)\big)  \\
= & D_X (-R(Z,W)Y) + \II(\nabla_X\nabla_Y Z, W) + \II(Z,\nabla_X\nabla_Y W) \\
& -D_Y (-R(Z,W)X) - \II(\nabla_Y\nabla_X Z, W) - \II(Z,\nabla_Y\nabla_X W) \\
= & -(D_X R)(Z,W)Y + (D_Y R)(Z,W)X -\II(R(X,Y)Z,W) - \II(Z, R(X,Y)W)
\end{align*}
Taking the tangent part yields $(\nabla_X R)(Z,W)Y = (\nabla_Y R)(Z,W)X$. Taking inner product with $T\in T_pM$ we have
\[(\nabla R) (Z,W,Y,T,X)=(\nabla R) (Z,W,X,T,Y),\]
that is, $\nabla R$ is symmetric in the third and fifth entries.
But $\nabla R$ is also skew-symmetric in the third and fourth entries, so that $\nabla R=0$.
\end{enumerate}
\end{proof}

\begin{lemma}
\label{L:unique}
Let $(M,g)$ be a connected Riemannian manifold, and let $\Omega_j:TM\to V_j$, for $j=1,2$ be virtual immersions with skew-symmetric second fundamental forms $\II_j$. Assume $V_1, V_2$ are minimal in the sense that $V_j=\operatorname{span}(\Omega_j (TM)) $. Then $\Omega_1, \Omega_2$ are equivalent in the sense that there is a linear isometry $L:V_1\to V_2$ such that $\Omega_2=L\circ\Omega_1$.
\end{lemma}
\begin{proof}
Define a connection $\hat{D}$ on the vector bundle $TM\oplus \wedge^2 TM$ by
\[ \hat{D}_W \Big(Z, \,\, \sum_i X_i\wedge Y_i\Big) =
\Big(\nabla_W Z -\sum_i R(X_i, Y_i)W,\,\,\, W\wedge Z+\nabla_W\sum_i X_i\wedge Y_i\Big)  \]
Define bundle homomorphisms $\hat\Omega_j :TM \oplus \wedge^2 TM \to M\times V_j$, for $j=1,2$, by
\[ \hat\Omega_j\Big(Z,\,\, \sum_i X_i\wedge Y_i\Big) = \Big(p, \,\,\, \Omega_j(Z)+\sum_i \II_j(X_i, Y_i)\Big) \]
for $Z,X_i,Y_i\in T_pM $. By Lemma \ref{L:locsym}(b), given vector fields $X_i, Y_i, Z, W$, we have
\begin{equation}
\label{E:connections}
(D_j)_W \Big( \hat\Omega_i\Big(Z,\quad \sum_i X_i\wedge Y_i\Big) \Big)= 
\hat\Omega_j \Big(\hat{D}_W \Big(Z,\quad \sum_i X_i\wedge Y_i\Big) \Big)
\end{equation}
where $D_j$ denotes the natural flat connection on $M\times V_j$. This implies that the image of $\hat\Omega_j$ is $D_j$-parallel, and hence, by minimality of $V_j$, that $\hat\Omega_j$ is onto $M\times V_j$.

Define a bundle isomorphism $L:M\times V_1\to M\times V_2$ by
\[ L \Big(\hat\Omega_1\Big(Z,\quad \sum_i X_i\wedge Y_i\Big)\Big) = 
\hat\Omega_2\Big(Z,\quad \sum_i X_i\wedge Y_i\Big)\]
for $Z,X_i,Y_i\in T_pM $. This is well-defined because, by Lemma \ref{L:locsym}(a), $\ker\hat\Omega_1=\ker\hat\Omega_2$. Indeed, they are both equal to 
\[ \Big\{ \Big. \Big(0,\quad \sum_i X_i\wedge Y_i\Big) \ \Big|\ X_i, Y_i\in T_pM,\ \sum_{a,b}R(X_a, Y_a, X_b, Y_b)=0\Big\} \]

We claim the linear map $L_p: V_1\to V_2$ is independent of $p\in M$. Indeed, given two points $p,q\in M$, choose a curve $\gamma(t)$ in $M$ joining $p$ to $q$. Choose smooth vector fields $Z,X_i,Y_i$ along $\gamma(t)$ such that $\hat\Omega_1(Z,\sum X_i\wedge Y_i)$ is constant equal to $v\in V_1$. Then, by \eqref{E:connections}, $\hat D_{\dot\gamma} (Z,\sum X_i\wedge Y_i)\In \ker\hat\Omega_1 $. But by Lemma \ref{L:locsym}(a), $\ker\hat\Omega_1=\ker\hat\Omega_2$. Therefore, again by \eqref{E:connections}, we see that $L(v)$ is constant along $\gamma$, so that $L_p=L_q$.
Calling this one linear map $L$, we have $\hat\Omega_2=L\circ\hat\Omega_1$ by construction. In particular, $\Omega_2=L\circ\Omega_1$, finishing the proof that $\Omega_1$ and $\Omega_2$ are equivalent.
\end{proof}

\begin{proof}[Proof of Theorem \ref{MT:classification}]
Let $(M,g)$ be a compact Riemannian manifold. If $M$ is symmetric, then it admits a virtual immersion with skew-symmetric second fundamental form by Lemma \ref{L:Omega0}. Conversely, let $\Omega:TM\to V$ be a  virtual immersion with skew-symmetric second fundamental form. We may assume that $V$ is minimal, and then uniqueness of $\Omega$ follows from Lemma \ref{L:unique}. What remains to be proved is that $M$ is symmetric.

By Lemma \ref{L:locsym}(c), $M$ is locally symmetric, and therefore its universal cover $\tilde{M}$ is symmetric. By Lemma \ref{L:locsym}(a), $\tilde{M}$ has non-negative curvature,  so that $\tilde{M}$ splits isometrically as $\tilde{M}= N\times \R^l$, where $N$ is a compact, simply-connected symmetric space.

Denoting by $G$ the isometry group of $N$, we claim that $\isom(\tilde{M})=G\times \isom(\R^l)$. Indeed, tangent vectors of the form $(v,0)$ are characterized by the fact that the associated geodesic has bounded image. Thus, any isometry $\gamma$ of $N\times \R^l$ preserves the splitting $T(N\times \R^l)=TN\oplus \R^l$. In particular, fixing $p\in N$, the maps $g:N\to N$ and $B:\R^l\to \R^l$ given by composing the obvious maps
\begin{align*}
g:N\to N\times\{0\}\hookrightarrow N\times \R^l \stackrel{\gamma}{\to} N\times \R^l \to N \\
B:\R^l\to \{p\}\times\R^l\hookrightarrow N\times \R^l \stackrel{\gamma}{\to} N\times \R^l \to \R^l 
\end{align*}
are isometries. Since $\gamma$ and $g\times B$ are isometries of $N\times \R^l$ whose values and first derivatives coincide at $(p,0)$, it follows that $\gamma=g\times B$.

Denote by $\Gamma\subset \isom(\tilde{M})=G\times \isom(\R^l)$ the group of deck transformations for the covering $\rho:\tilde{M}\to M$ (so that $\Gamma$ is isomorphic to $\pi_1(M)$).
Let $(p,\xi)\in N\times \R^l$, and consider the symmetry $s=s_p\times s_\xi$ at $(p,\xi)$. We need to show that $s$ normalizes $\Gamma$, so that $s$ descends to a well-defined symmetry of $M$. We will in fact show that $s\gamma s=\gamma^{-1}$ for every $\gamma=g\times B\in \Gamma$.

Note that $\Omega_0\times\id:T\tilde{M}\to \g\times\R^l$ is a virtual immersion with skew-symmetric second fundamental form, where $\Omega_0$ is defined as in \eqref{E:naturalgenimm}. Since the pull-back $\rho^*\Omega$ is again such a virtual immersion,
 Lemma \ref{L:unique} implies that $\Omega_0\times\id$ is fixed by $\Gamma$. Therefore, $B$ must be a translation. As for $g\in G$, we have $\Ad_g X=X$ for every $X\in \m$, and, since $[\m,\m]=\h$, the same equation holds for every $X\in\g$. In particular, $g$ commutes with the identity component $G_0$ of $G$. 

Since $G_0$ acts transitively on $N$, this implies that the displacement function $q\in N \mapsto d(q, g (q))$ has a constant value $d$ (in other words, it is a Clifford-Wolf translation). Moreover, the isometries $s_p g s_p$ and $g s_p g s_p$ also commute with $G_0$, so that they are Clifford-Wolf translations as well.  Thus it suffices to show that $g s_p g s_p$ has a fixed point, for this would imply that $g s_p g s_p=\id$, and hence that $\gamma s \gamma s= \id$.

Let $c(t):[0,1]\to N$ be a minimal geodesic between $c(0)= p$ and $c(1)=g(p)$, and let $m=c(1/2)$ denote the midpoint. Then, the concatenation of $c$ with $g^{-1}c$ must be a geodesic, because $d(g^{-1}m,m)=d=d(g^{-1}m,p)+d(p,m)$. Thus, extending the geodesic segment $c$ to a complete geodesic $c:\R\to N$, we see that $c(-1)=g^{-1}(p)$.
In particular, $s_p g s_p (p)= s_p g (p)= s_p (c(1))= c(-1) = g^{-1}(p)$, and hence $g s_p g s_p$ fixes the point $p\in N$, finishing the proof.
\end{proof}

\section{Affine bounds on the index}

Let $\Sigma\to M=G/H$ be a compact, immersed, minimal hypersurface in a compact symmetric space. This section addresses the question of when the linear bound in $b_1(\Sigma)$
 on the extended index of $\Sigma$ given in Theorem \ref{MT:indexbound} can be ``improved'' to an affine bound on the index.  
To find such affine bounds, we consider the unique virtual immersion $\Omega:TM\to \mathfrak{g}$ with skew-symmetric second fundamental form, defined in \eqref{E:naturalgenimm}. In view of Proposition \ref{P:affinebound}, it suffices to find an upper bound on the dimension of the space of harmonic 1-forms $\omega$ on $\Sigma$ such that $X_{ij}(\omega)$ lies in the kernel of the Jacobi operator $J_\Sigma$ for all $i,j$, where $X_{ij}$ is defined in \eqref{E:variation}.

To compute $J_\Sigma(X_{ij}(\omega))$ at $p\in \Sigma$, choose an orthonormal frame $E_1,\ldots E_{n-1}$ of $\Sigma$, such that $(\nabla^{\Sigma}_{E_i}E_j)_p=0$, and an orthonormal basis $\theta_i$ of $\mathfrak{g}$. Then
\begin{align*}
J_\Sigma(X_{ij}(\omega))=&\nabla^{\perp}_{E_k}\nabla^{\perp}_{E_k}\big(\scal{N\wedge \omega^\#,\theta_i\wedge \theta_j}N\big)\\
&+ (|A|^2+Ric_M(N,N))\scal{N\wedge \omega^\#,\theta_i\wedge \theta_j}N \\
=& \big<D_{E_k}D_{E_k}(N\wedge \omega^\#)+(|A|^2+Ric_M(N,N))(N\wedge \omega^\#),\theta_i\wedge \theta_j\big> N.
\end{align*}
Here and in the rest of the section we adopt the convention that, when repeated indices appear, we are  summing over them. From the previous equation $J_\Sigma(X_{ij}(\omega))=0$ for all $i,j$ if and only if
\begin{equation}\label{E:nullity}
D_{E_k}D_{E_k}(N\wedge \omega^\#)+(|A|^2+Ric_M(N,N))(N\wedge \omega^\#)=0.
\end{equation}

We proceed now to compute Equation \eqref{E:nullity}. For this, let $\nabla$ denote the Levi Civita connection of $\Sigma$. We have:
\begin{align}
\label{E:the-eqn} D_{E_k}D_{E_k}(N\wedge {\omega^\#})= D_{E_k}&\Big(-S_NE_k\wedge {\omega^\#}+ \II(E_k,N)\wedge {\omega^\#}\\& +N\wedge \nabla_{E_k}{\omega^\#}+N\wedge \II(E_k,{\omega^\#})\Big)\nonumber
\end{align}
We compute the derivatives of each of the four summands on the right hand side of the equation above, in (a)--(d) below. For the sake of clarity, when computing terms of the type $D_{E_k}(X\wedge Y)$, we display the result in the form $((D_{E_k}X)\wedge Y)+(X\wedge (D_{E_k}Y))$, that is, we write the two parts separately inside parentheses.
\begin{align}
\label{a}\tag{a} D_{E_k}(-S_NE_k\wedge {\omega^\#})=& \Big(-\nabla_{E_k}(S_NE_k)\wedge {\omega^\#} -|S_N|^2 N\wedge {\omega^\#}\\
&-\II(E_k,S_N{E_k})\wedge {\omega^\#}\Big)+\Big(-S_NE_k\wedge \nabla_{E_k}{\omega^\#}\nonumber \\
& - S_NE_k\wedge\left<S_N{\omega^\#}, E_k\right>N-S_NE_k\wedge \II(E_k,{\omega^\#})\Big) \nonumber
\end{align}
\begin{align}
\label{b}\tag{b} D_{E_k}(\II(E_k,N)\wedge {\omega^\#})=&\Big(-R(E_k,N)E_k\wedge {\omega^\#} -\II(E_k,S_NE_k)\wedge {\omega^\#}\Big)\\
&\!+\!\Big(\II(E_k,N)\!\wedge\! \nabla_{E_k}{\omega^\#}+\II(E_k,N)\!\wedge\! \left<S_NE_k,{\omega^\#}\right>\!N\nonumber \\
&+\II(E_k,N)\wedge\II(E_k,{\omega^\#})\Big)\nonumber
\end{align}
In the equation above, it was used the fact that $(D_{Z}\II)(X,Y)=-R(X,Y)Z$, and that by assumption $\nabla_{E_k}E_j=0$ at $p$.
\begin{align}
\label{c}\tag{c} D_{E_k}(N\wedge \nabla_{E_k}{\omega^\#})=&\Big(-S_NE_k\wedge \nabla_{E_k}{\omega^\#}+\II(E_k,N)\wedge  \nabla_{E_k}{\omega^\#}\Big)\\
&+\Big(N\wedge \nabla_{E_k}\nabla_{E_k}{\omega^\#}+  N\wedge \II(E_k,\nabla_{E_k}{\omega^\#})\Big)\nonumber
\end{align}

\begin{align}
\label{d}\tag{d} D_{E_k}(N\wedge \II(E_k,{\omega^\#}))=&\Big(-S_NE_k\wedge \II(E_k,{\omega^\#})+ \II(E_k,N)\wedge \II(E_k,{\omega^\#})\Big)\\
&+\Big(-N\wedge R(E_k,{\omega^\#})E_k+N\wedge\II(E_k,\nabla_{E_k}{\omega^\#})\nonumber\\
&+ N\wedge \II(E_k, \left<S_N{\omega^\#}, E_k\right>N)\Big)\nonumber
\end{align}

\begin{lemma}
\label{L:rigidity}
Let $M=G/H$ be a compact symmetric space and let $\Sigma\to M$ a closed, immersed, minimal hypersurface. Suppose $\omega$ is a harmonic one-form on $\Sigma$ such that $J_\Sigma(X_{ij}(\omega))=0$ for all $1\leq i<j\leq d$. Then
\begin{enumerate}[a)]
\item The operators $\nabla\omega^\#$ and $S_N$ commute.
\item At any point $p=\llbracket g\rrbracket\in G/H$, $Ad_g^{-1}\II(\omega^\#,N)$ is contained in the center of $\mathfrak{h}$, $\mathfrak{z}(\mathfrak{h})=\{x\in \mathfrak{h}\mid [x,y]=0\,\forall y\in \mathfrak{h}\}$.
\item For any vector $x$ tangent to $\Sigma$, $\nabla_x\left(\II(\omega^\#,N)\right)=-R(\omega^\#,N)x$. In particular, $\|\II(\omega^\#,N)\|$ is constant.
\end{enumerate}
\end{lemma}

For example, when the symmetric space $M$ is a torus, parts (b) and (c) are trivially satisfied, and part (a) is equivalent to Proposition 3 in \cite{ACS17}.

\begin{remark}
In fact, $J_\Sigma(X_{ij}(\omega))=0$ for all $1\leq i<j\leq d$ if and only if conditions (a), (b), (c) are satisfied. We omit the proof of the reverse implication since it is not used in the remainder of this article.
\end{remark}

\begin{proof} Notice that Equation \eqref{E:nullity} is a vector-valued equation in $\wedge^2 V$. Fixing a point $p$ in $\Sigma$ and identifying $T_pM$ with its image under $\Omega_p$ in $V$, we can split orthogonally $V=T_pM\oplus \nu_pM=\R\cdot N\oplus T_p\Sigma\oplus \nu_p M$
and this induces a splitting of $\wedge^2V$. The different parts of the lemma follow from projecting Equation \eqref{E:nullity} on the different subspaces.
\\

a) Projecting Equation \eqref{E:nullity} onto the subspace $\wedge^2T_p\Sigma\subset \wedge^2 V$ and using the computations above, we obtain
\begin{align} \label{E:eqn a}
0=&-\nabla_{E_k}(S_NE_k)\wedge {\omega^\#} -2S_NE_k\wedge \nabla_{E_k}{\omega^\#} -\pi_\Sigma(R(E_k,N)E_k)\wedge {\omega^\#}.
\end{align}
Here $\pi_\Sigma$ denotes orthogonal projection onto $T_p\Sigma$. Applying Codazzi equation to the first term we get
\begin{align*}
\nabla_{E_k}(S_NE_k)=&\left<\nabla_{E_k}(S_NE_j),E_k\right>E_j\\
=&\left(\left<\nabla_{E_j}(S_NE_k),E_k\right> + \left<R(E_j,E_k)E_k,N\right>\right)E_j\\
=&\left(E_j\left<S_NE_k,E_k\right>-\left<R(E_k, N)E_k,E_j\right>\right)E_j\\
=&-\pi_\Sigma(R(E_k, N)E_k)
\end{align*}
The last equation holds because, since $\Sigma$ is minimal, the first summand in the second equation vanishes. Equation \eqref{E:eqn a} thus becomes
\begin{align*}
0=&\pi_\Sigma(R(E_k,N)E_k)\wedge {\omega^\#} -2S_NE_k\wedge \nabla_{E_k}{\omega^\#} -\pi_\Sigma(R(E_k,N)E_k)\wedge {\omega^\#}\\
\Rightarrow 0=&S_NE_k\wedge \nabla_{E_k}{\omega^\#}
\end{align*}
For any $x,y\in T_p\Sigma$, we thus have
\begin{align*}
0=&\left<S_NE_k\wedge \nabla_{E_k}{\omega^\#},x\wedge y\right>\\
=&\left<S_NE_k,x\right>\left<\nabla_{E_k}{\omega^\#},y\right>-\left<S_NE_k,y\right>\left<\nabla_{E_k}{\omega^\#},x\right>
\end{align*}
Clearly $S_N$ is symmetric. Since $\omega$ is harmonic, $\nabla\omega^\#$ is symmetric as well, and the equation above becomes
\begin{align*}
0=&\left<E_k,S_Nx\right>\left<\nabla_{y}{\omega^\#},E_k\right>-\left<E_k,S_Ny\right>\left<\nabla_{x}{\omega^\#},E_k\right>\\
=&\left<S_Nx,(\nabla\omega^\#)y\right>-\left<S_Ny,(\nabla\omega^\#)x\right>\\
=&\left<[S_N, \nabla\omega^\#]x,y\right>
\end{align*}
Since this holds for every $x$ and $y$, the result follows.
\\

b) Projecting Equation \eqref{E:nullity} onto the subspace $\wedge^2\nu_pM\subset \wedge^2 V$ one gets
\begin{align}\label{E:eqn b}
\II(N,E_k)\wedge\II({\omega^\#}E_k)=0.
\end{align}
Taking inner product with elements $\II(x,y)\wedge\II(u,v)$ (such elements span $\wedge^2\nu_pM$), one gets
\begin{align*}
0=&\left<\II(x,y),\II(N,E_k)\right>\left<\II(u,v),\II(\omega^\#,E_k)\right>-\left<\II(x,y),\II(\omega^\#,E_k)\right>\left<\II(u,v),\II(N,E_k)\right>\\
=&\left<R(x,y)N,E_k\right>\left<R(u,v)\omega^\#,E_k\right>-\left<R(x,y)\omega^\#,E_k\right>\left<R(u,v)N,E_k\right>\\
=&\left<R(x,y)N, R(u,v)\omega^\#\right>-\left<R(x,y)\omega^\#,R(u,v)N\right>\\
\end{align*}
It is easy to check that, given $\eta=\II(x,y)$, then $S_\eta=R(x,y)$. The equation above then becomes
\[
0=\left<[S_{\eta_1},S_{\eta_2}]N,\omega^\#\right>, \quad \eta_1=\II(x,y),\,\eta_2=\II(u,v)
\]
Since $\Omega(TM)=\g$, equation above holds for any $\eta_1,\,\eta_2$ normal vectors. Using Ricci equation, this implies that
\[
\left<R^\perp(\omega^\#,N)\eta_1,\eta_2\right>=0
\]
and in particular $R^\perp(\omega^\#,N)\eta=0$ for all $\eta$ in $\nu_pM$. From Lemma \ref{L:Omega0}, letting $p=\llbracket g\rrbracket $, then $\nu_pM=Ad_g\mathfrak{h}$ and, letting $\eta=Ad_g v$ for $v\in \mathfrak{h}$, one has
\[
[\II(\omega^\#,N),\eta]=0\quad \forall \eta\in Ad_g\mathfrak{h}\quad\Rightarrow\quad [Ad_{g}^{-1}\II(\omega^\#,N),v]=0\quad \forall v\in \mathfrak{h}
\]
Therefore $Ad_{g}^{-1}\II(\omega^\#,N)$ belongs to the center of $\mathfrak{h}$.
\\

c) Projecting Equation \eqref{E:nullity} onto the subspace $T_p\Sigma\otimes \nu_pM\subset \wedge^2 V$ we get
\begin{align}\label{E:eqn c}
0&=-2(\II(E_k,S_N{E_k})\wedge {\omega^\#}+S_NE_k\wedge \II(E_k,{\omega^\#})-\II(E_k,N)\wedge \nabla_{E_k}{\omega^\#}).
\end{align}
The first term can be rewritten as $\II(E_k,E_j)\left<S_NE_k,E_j\right>\wedge {\omega^\#}$. However, the left factor of the exterior product is skew symmetric in $i,k$ and therefore the $(j,k)$-term in the sum cancels with the $(k,j)$-term. This term thus vanishes, and Equation \eqref{E:eqn c} is equivalent to
\[
S_NE_k\wedge \II(E_k,{\omega^\#})+\nabla_{E_k}{\omega^\#}\wedge\II(E_k,N)=0.
\]
Taking the inner product with an element of the form $x\wedge \II(y,z)$ for $x\in T_p\Sigma$ and $y,z\in T_pM$ (these elements span the whole of $T_p\Sigma\otimes \nu_pM$) one gets
\begin{align*}
0=&\left<S_NE_k\wedge \II(E_k,{\omega^\#})+\nabla_{E_k}{\omega^\#}\wedge\II(E_k,N), x\wedge \II(y,z)\right>\\
=&\left<S_NE_k,x\right>\left<\II(E_k,\omega^\#),\II(y,z)\right>+\left<\nabla_{E_k}{\omega^\#},x\right>\left<\II(E_k,N),\II(y,z)\right>\\
=&-\left<S_Nx,E_k\right>\left<R(y,z)\omega^\#,E_k\right>-\left<\nabla_{x}{\omega^\#},E_k\right>\left<R(y,z)N,E_k\right>\\
=&-\left<S_Nx,R(y,z)\omega^\#\right>-\left<\nabla_{x}{\omega^\#},R(y,z)N\right>\\
=&-\left<R(\omega^\#, S_Nx)y,z\right>+\left<R(\nabla_{x}{\omega^\#},N)y,z\right>\\
=&\left<R(\omega^\#, \nabla_xN)y+R(\nabla_{x}{\omega^\#},N)y,z\right>\\
=&\left<\II(\omega^\#, \nabla_xN)+\II(\nabla_{x}{\omega^\#},N), \II(y,z)\right>
\end{align*}

Therefore it follows that
\[
0=\II(\omega^\#, \nabla_xN)+\II(\nabla_{x}{\omega^\#},N)=\nabla_x(\II(\omega^\#,N))-R(\omega^\#,N)x.
\]
\end{proof}

\begin{theorem}
\label{T:rigidity2}
Let $M=G/H$ be a compact symmetric space, $\Sigma\to M$ a closed, immersed, minimal hypersurface, and let 
$\Hr_1$ denote the space of harmonic one-forms $\omega$ on $\Sigma$ satisfying the following two conditions at all points $p\in \Sigma$.
\begin{enumerate}[a)]
\item The operators $\nabla\omega^\#$ and $S_N$ commute.
\item $R(\omega^\#,N,\omega^\#,N)=0$. 
\end{enumerate}
Then
\[ \Ind(\Sigma)\geq {d\choose 2}^{-1}\big(b_1(\Sigma)-\dim \Hr_1-\dim\mathfrak{z}(\mathfrak{h})\big)\]
where $\mathfrak{z}(\mathfrak{h})$ denotes the center $\mathfrak{z}(\mathfrak{h})=\{x\in \mathfrak{h}\mid [x,y]=0\,\forall y\in \mathfrak{h}\}$.
\end{theorem}
\begin{proof}
Let $\Hr_0$ be the space of harmonic one-forms on $\Sigma$ that satisfy the three conditions listed in Lemma \ref{L:rigidity}. Then $\Hr_1$ is a subspace of $\Hr_0$. More precisely, 
letting $p_0=[e]\in M=G/H$, we may assume $p_0\in \Sigma$  and define the linear map $Z:\Hr_0\to \mathfrak{z}(\mathfrak{h})$ by $Z(\omega)=\II_{p_0}(\omega^\#,N)$. By condition (c) in Lemma \ref{L:rigidity}, $\Hr_1$ equals the kernel of $Z$. Therefore its codimension is at most $\dim\mathfrak{z}(\mathfrak{h})$, and the result follows from Lemma \ref{L:rigidity} and Proposition \ref{P:affinebound}.
\end{proof}

\begin{corollary}
\label{C:CROSS}
Let $M^n=G/H$, $n>2$, be a CROSS, and $\Sigma\to M$ a compact, immersed, minimal hypersurface.
Then
\[\Ind(\Sigma)\geq {\dim G\choose 2}^{-1}b_1(\Sigma).\]
\end{corollary}
\begin{proof}
The CROSSes $M=G/H$ are $S^n=\SO(n+1)/\SO(n)$, $\mathbb{RP}^n=\OO(n+1)/\OO(n)$, $\mathbb{CP}^n=\SU(n+1)/\SU(n)$, $\mathbb{HP}^n=\Sp(n+1)/\Sp(n)$ and $Ca\mathbb{P}^2=F_4/\Spin(9)$. In all these cases,  $\mathfrak{h}$ is semisimple and hence centerless. Now the result follows from the fact that $M$ has positive sectional curvature, together with Theorem \ref{T:rigidity2}.
\end{proof}

\begin{remark}
\label{R:2dCROSS}
The only $2$-dimensional CROSSes are $S^2$ and $\mathbb{RP}^2$. In these cases, a minimal hypersurface $\Sigma$ is a closed geodesic, so that $b_1(\Sigma)=1$. Since the index is non-negative, in particular one has $\Ind(\Sigma)\geq b_1(\Sigma)-1$.
\end{remark}

%

\begin{proof}[Proof of Theorem \ref{MT:affine}(b)]
Define $\delta_i$ for $i=1,2$ by: $\delta_i=1$ if $\dim M_i=2$, and $\delta_i=0$ otherwise. Note that $D=1+\delta_1+\delta_2$, and that $\delta_1+\delta_2$ is the dimension of the center of $\h$.

Suppose first that $\Sigma$ is of the form $M_1\times \Sigma_2$, where $\Sigma_2$ is a minimal, compact hypersurface in $M_2$. By Corollary \ref{C:CROSS} and Remark \ref{R:2dCROSS}, the index of $\Sigma_2$ is bounded below by ${\dim G_2\choose 2}^{-1}(b_1(\Sigma)-\delta_2)$, In this case we have $b_1(\Sigma)=b_1(\Sigma_2)$, and
\begin{align*}
\Ind(\Sigma)\geq \Ind(\Sigma_2)&\geq {\dim G_1\choose 2}^{-1}(b_1(\Sigma_2)-\delta_2)\\
&\geq {\dim G_1\times G_2\choose 2}^{-1}(b_1(\Sigma)-1-\delta_1-\delta_2)
\end{align*}
thus the proposition is proved in this case. Clearly the same argument would work in the case $\Sigma=\Sigma_1\times M_2$, where $\Sigma_1$ is a minimal compact immersed hypersurface in $M_1$.

Suppose now that $\Sigma$ is not as before. 
Then, by Theorem \ref{T:rigidity2}, it suffices to show that the space of harmonic one-forms $\omega$ on $\Sigma$ such that $R(\omega^\#,N,\omega^\#,N)=0$ is at most one-dimensional. 


Let $\omega_1,\omega_2$ be two such harmonic one-forms.
Since $\Sigma$ is neither of the form $\Sigma_1\times M_2$ nor of the form $M_1\times \Sigma_2$, there exists an open set $U\subset \Sigma$ such that for every $p=(p_1,p_2)\in U$, the normal vector $N_p\in T_{p_1}M_1\oplus T_{p_2}M_2$ is not tangent to neither $T_{p_1}M_1$ nor $T_{p_2}M_2$. In particular, for every $p\in U$ there exists a unique zero curvature plane $\pi_p$ through $N_p$, and in particular a unique direction in $\pi_p$, perpendicular to $N_p$. It thus follows that $\omega_1^\#$ and $\omega_2^\#$ are collinear in $U$: $\omega_1=f\omega_2$ for some function $f:U\to \R$. However, since $\omega_1,\omega_2$ are both closed and co-closed, it is easy to check that $df$ must be both parallel and normal to $\omega_2$ in $U$, and in particular $df=0$. Thus $f$ is constant, and $\omega_1,\omega_2$ are linearly dependent in $U$, hence linearly dependent on the whole of $\Sigma$, a contradiction.
\end{proof}

\begin{proof}[Proof of Theorem \ref{MT:affine}(a)]
By Theorem \ref{T:rigidity2}, it suffices to show that $\dim\Hr_1\leq 2r-3$. Let $\omega\in \Hr_1$. The fact that $\nabla\omega^\#$ and $S$ commute is equivalent to condition (2.3) in \cite[Proposition 5]{ACS17}, and from that proposition it follows that $\omega$ is determined by its value and the value of its covariant derivative at any point $p$.

Let $p\in \Sigma$ where all the principal curvatures are distinct. We may assume that $N(p)\in T_pM$ belongs to a regular (that is, principal or exceptional) orbit under the isotropy representation of $G_p$ on $T_pM$. Indeed, the set of singular vectors has codimension at least two, and, since the shape operator has distinct eigenvalues, its image has codimension at most one. This means that if $N(p)$ is singular, then there is a nearby $p'\in \Sigma$ such that $N(p')$ is regular.

Define
\[\mathcal{V}=\{X\in T\Sigma \ |\ \II(X,N)=0 \}=\{X\in T\Sigma \ |\ R(X,N,X,N)=0 \}.\]
Since $N(p)$ is regular, $\mathcal{V}$ is a smooth distribution of rank $r-1$ on an open subset of $\Sigma$ containing $p$. Moreover, for any $\omega\in\Hr_1$, we have $\omega^\#\in\mathcal{V}$.

Let $\{e_1, \ldots, e_n\}$ be an orthonormal frame of eigenvectors for the shape operator, defined on an open neighbourhood of $p$, with $S(e_i)=a_i e_i$. Since $\nabla \omega^\#$ commutes with $S$, there are functions $\lambda_i$ defined near $p$, such that $\nabla_{e_i} \omega^\#=\lambda_i e_i$. Differentiating the equation $\II(\omega^\#,N)=0$ and using Lemma \ref{L:locsym}, we obtain
\[ \lambda_i\II(e_i, N)=-a_i\II(\omega^\#, e_i) .\]
Since $N$ is regular, there are $k\leq r-1$ values of $i$ such that $\II(e_i,N)=0$. Assume without loss of generality that $\II(e_{i},N)\neq 0$ for any $i> k$. From the equation above, $\lambda_i$ is completely determined by $\omega^\#$ for any $i>k$. The values $(\lambda_1(p),\ldots \lambda_k(p))$ are free, up to the further constraint that $\sum_{i=1}^n\lambda_i(p)=0$.

Therefore, the pair $(\omega^\#(p), \nabla\omega^\#(p))$ is determined by $(\omega^\#(p),\lambda_1(p),\ldots \lambda_{k-1}(p))\in \V\times \R^{k-1}$. In particular $\dim\Hr_1\leq r+k-2\leq 2r-3$.
\end{proof}

\bibliographystyle{alpha}

\end{document}